\documentclass[12pt]{article}
\usepackage{amsthm}
\newtheorem{theorem}{Theorem}
\newtheorem{lemma}{Lemma}
\newcommand\argmin{\mathop{\rm argmin}}
\newcommand\eref[1]{$(\ref{#1})$}
\newcommand\sspan{\mathop{\rm span}}
\renewcommand\b{{\mathbf{b}}}

\newcommand\g{{\mathbf{g}}}
\newcommand\m{{\mathbf{m}}}
\newcommand\n{{\mathbf{n}}}
\newcommand\p{{\mathbf{p}}}
\renewcommand\r{{\mathbf{r}}}
\newcommand\R{{\mathbf{R}}}
\renewcommand\u{{\bf u}}
\renewcommand\v{{\bf v}}
\newcommand\w{{\bf w}}
\newcommand\x{{\bf x}}
\newcommand\y{{\bf y}}

\newcommand\commentout[1]{\relax}
\title{A new secant method for unconstrained optimization (draft)}
\author{Stephen A.~Vavasis\thanks{Department of Combinatorics
and Optimization, University of Waterloo, 200 University Ave.\ W.,
Waterloo, Ontario, Canada, N2L 3G1, {\tt vavasis@math.uwaterloo.ca}.
Supported in part by a grant
from the Natural Sciences and Engineering Research Council (NSERC) of
Canada.}}
\begin{document}
\maketitle
\begin{abstract}
We present a gradient-based algorithm for unconstrained minimization
derived from iterated linear change of basis.  
The new method is equivalent to linear
conjugate gradient in the case of a quadratic objective function.
In the case of exact line search it is a secant method.
In practice, it performs comparably to
BFGS and DFP and is sometimes more robust.
\end{abstract}
\section{Iterated linear change of basis}
We consider the problem of minimizing a differentiable function
$f:\R^n\rightarrow \R$ with no constraints on the variables.
We propose the following algorithm for this problem. 
We assume
a starting point $\w_{{0}}$ is given.  Let $f_{{0}}$ be identified
with $f$.
\begin{tabbing}
++++\=++\=++\=++\=\kill 
\> {\bf Algorithm 1} \\
{[1]}\>for $k=1,2,\ldots$ \\
{[2]}\>\>$\p_{{k}} := -\nabla f_{{k-1}}(\w_{{k-1}})$; \\
{[3]}\>\>$\alpha_{{k}}:=\argmin\{f_{{k-1}}(\w_{{k-1}}+\alpha\p_{{k}}):\alpha\ge 0\}$; \\
{[4]}\>\>$\tilde\w_{{k}}:=\w_{{k-1}}+\alpha_{{k}}\p_{{k}}$; \\
{[5]}\>\>$\g_{{k}}:=-\nabla f_{{k-1}}(\tilde\w_{{k}})$; \\
{[6]}\>\>Define $l_{{k}}:\R^n\rightarrow\R^n$ by
$l_{{k}}(\x)=(I+\p_{{k}}\g_{{k}}^T/\Vert\p_{{k}}\Vert^2)(\x)$; \\
{[7]}\>\>$f_{{k}} := f_{{k-1}}\circ l_{{k}};$ \\
{[8]}\>\>$\w_{{k}}:=l_{{k}}^{-1}(\tilde\w_{{k}})$; \\
{[9]}\> end
\end{tabbing}
Lines [1]--[4] of this algorithm are the standard steepest descent
computation.
In the third line, an inexact line search may be used in place
of exact minimization over $\alpha$.  In the sixth line, $I$ is
the $n\times n$ identity matrix.
The seventh line indicates
functional composition: a new objective function is formed
as the composition of the old objective function and a linear
change of variables.  

The eighth line applies the inverse transformation
to $\tilde\w_{{k}}$ so as to enforce the relationship
$f_{{k-1}}(\tilde\w_{{k}})=f_{{k}}(\w_{{k}})$. 
The inverse transform is efficiently computed and applied using the
Sherman-Morrison formula.
Although the function value is invariant, the
gradient value is not, so the algorithm is not equivalent to a
sequence of 
steepest descent steps in the original coordinates.
This algorithm is equivalent to the linear conjugate gradient 
algorithm in the case that $f$ is a convex
quadratic function and the line search
is exact, as we shall see in Section~\ref{sec:linearcg}.

When the algorithm terminates, say at
iteration $N$, the vector $\w^{(N)}$ is a minimizer or
approximate minimizer of $f^{(N)}$.  Therefore, the linear transformations
must be saved and applied to $\w^{(N)}$ in order to recover a solution
to the original problem.

In certain special classes of problems, it may be feasible to implement
the algorithm exactly as stated because the objective function may
be accessible for updating.  More commonly, however, the objective function
is available only as a subroutine, in which case the algorithm
must be restated in a way so that it keeps track of the linear updates
itself.  In particular, it must save the two vectors defining the
linear transformation from all previous iterations.  
Then the chain rule is applied, which states that if $g(\x)=f(l(\x))$,
where $l$ is a linear function, then $\nabla g(\x)=l^T(\nabla f(l(\x)))$,
where $l^T$ denotes the transposed linear function.
This version of
the algorithm is as follows.  There is no longer a subscript on
$f$ since $f$ is not explicitly updated in this version.
The current iterate in this algorithm is denoted $\x_{{k}}$ and
must be initialized as $\x_{{0}}$, which is equal to $\w_{{0}}$ in
Algorithm 1.
\begin{tabbing}
++++\=++\=++\=++\=\kill 
\> {\bf Algorithm 2} \\
{[1]}\>for $k=1,2,\ldots$ \\
{[2]}\>\>$\p_{{k}} := -l_{{k-1}}^T\circ\cdots\circ l_{{1}}^T(\nabla f(\x_{{k-1}}))$; \\
{[3]}\>\>$\m_{{k}} := l_{{1}}\circ\cdots\circ l_{{k-1}}(\p_{{k}})$; \\
{[4]}\>\>$\alpha_{{k}}:=\argmin\{f(\x_{{k-1}}+\alpha\m_{{k}}):\alpha\ge 0\}$; \\
{[5]}\>\>$\x_{{k}}:=\x_{{k-1}}+\alpha_{{k}}\m_{{k}}$; \\
{[6]}\>\>$\g_{{k}}:= -l_{{k-1}}^T\circ\cdots\circ l_{{1}}^T(\nabla f(\x_{{k}}))$; \\
{[7]}\>\>Define $l_{{k}}:\R^n\rightarrow\R^n$ by
$l_{{k}}(\x)=(I+\p_{{k}}\g_{{k}}^T/\Vert\p_{{k}}\Vert^2)\x$; \\
{[8]}\> end
\end{tabbing}
The fact that Algorithms 1 and 2 are equivalent is an easy
induction.  The variables $\p_{{k}}$, $\alpha_{{k}}$ and $\g_{{k}}$
are identical between the two algorithms, as are the sequences of
linear transformations $l_{{k}}$.  The remaining variables have the
following relationships: $\x_{{k}}=l_{{1}}\circ\cdots\circ l_{{k}}(\w_{{k}})$
and $\x_{{k}}=l_{{1}}\circ\cdots\circ l_{{k-1}}(\tilde\w_{{k}})$.
Note that some redundant computation in step [2] can be saved
by observing that 
\begin{equation}
\p_{{k}}=l_{{k-1}}^T(\g_{{k-1}}),
\label{eq:pkgk}
\end{equation}
where $\g_{{k-1}}$ was
computed in step [6] of the previous iteration.

We conclude this section with a result concerning the invertibility
of the linear transformations.

\begin{lemma}
Assume $f$ is $C^1$ and none of the iterates in Algorithm $1$ is 
a stationary point.
Suppose an exact line search is used in Algorithm $1$.  Then
$l_k$ is invertible on every step.
\end{lemma}
\begin{proof}
It follows from standard theory of steepest descent that if an
exact line search is used, then 
\begin{equation}
\p_k^T\g_k=0.
\label{eq:pgorthog}
\end{equation}
(This is the first-order
condition for the optimality of $\alpha$ for the differentiable
function $f(\w_{k-1}+\alpha\p_k)$).  In this case,
$I+\p_k\g_k^T/(\p_k^T\p_k)$ is invertible since it follows from the
Sherman-Morrison formula that $I+\u\v^T$ is invertible unless $\u^T\v=-1$.
\end{proof}

We remark that many kinds of inexact line searches will also 
yield the same result.  The requirement for invertibility of $l_k$ is
that $\p_k^T\p_k\ne -\g_k^T\p_k$.  Written in terms of the line search
function $\phi(\alpha)=f_k(\w_k+\alpha\p_k)$, this is the same as
saying that $\phi'(0)\ne-\phi'(\alpha_k)$.  A line search will often
enforce the condition $|\phi'(\alpha)|<|\phi'(0)|$.

Steps [2]--[3] of Algorithm 2 may be written as
$\m_k=-H_k\nabla f(\x_{k-1})$, where $H_k=l_1\cdots l_{k-1}l_{k-1}^T\cdots l_1^T$.
Obviously, $H_k$ is positive semidefinite, and assuming the condition
in the previous paragraph holds, it is positive definite.  This means
that Algorithm 2 always produces descent directions except in the 
unexpected case that $\p_k^T\p_k= -\g_k^T\p_k$.

\section{Specialization to quadratic functions}
In this section we present some results on the specialization of 
Algorithm 1
to convex quadratic functions with exact line search.  In particular,
we prove finite termination of the algorithm.  Finite termination
is also a consequence of the equivalence to linear conjugate gradient
(discussed in the next section), but the proof presented here
is a short self-contained proof that follows different lines
from customary proofs
of finite termination.
The difference
arises from the fact that Algorithm 1 is a one-step method (i.e.,
it does not involve recurrences), and therefore its analysis does not 
require
an induction hypothesis that spans the iterations as in the customary
analysis.

Suppose that $f(\w)=f_0(\w)=\w^TA_0\w/2 - \b_0^T\w$, where 
$A_0\in\R^{n\times n}$ is symmetric and positive definite.
Then it follows from step [7] of Algorithm 1 that
$f_k(\w)=\w^TA_k\w/2-\b_k^T\w$, where $A_k=l_k^T\cdots l_1^TA_0l_1\cdots l_k$
and $\b_k=l_k^T\cdots l_1^T\b_0$.

In the case of quadratic functions,
the optimal choice of $\alpha_k$  in step [3] of Algorithm 1
is well known to be (see \cite{GVL})
\begin{equation}
\alpha_k=\frac{\p_k^T\p_k}{\p_k^TA_{k-1}\p_k}.
\label{eq:alpha}
\end{equation}
We can develop the following further relationships.  
Combining \eref{eq:pkgk} and \eref{eq:pgorthog} yields
\begin{eqnarray*}
\p_{k+1} &=& l_k^T(\g_k) \\
& = & \left(I+\frac{\g_k\p_k^T}{\p_k^T\p_k}\right)\g_k \\
&=&\g_k.
\end{eqnarray*}
Also,
\begin{eqnarray*}
\g_k=\p_{k+1} &= & -\nabla f_{k-1}(\tilde\w_k) \\
&=& -A_{k-1}\tilde\w_k+\b_{k-1} \\
&=& -A_{k-1}\w_k+\b_{k-1}-\alpha_kA_{k-1}\p_k \\
& = & \p_k-\alpha_kA_{k-1}\p_k.
\end{eqnarray*}

With these relationships in hand, we can now propose the main
result that implies to finite termination.

Let us introduce the following notation:
$$K(A_{k-1},\p_{k})=\sspan(\p_k,A_{k-1}\p_k,A_{k-1}^2\p_k,\ldots),$$
i.e., the minimal invariant subspace of $A_{k-1}$ that contains $\p_k$.

\begin{theorem}
The invariant subspace $K(A_{k},\p_{k+1})$ is a proper subspace
of $K(A_{k-1},\p_k)$.
\end{theorem}
\begin{proof}
First, observe that $\p_{k+1}(=\g_k)\in K(A_{k-1},\p_{k})$, which follows
from equality demonstrated above that $\p_{k+1}=\p_k-A_{k-1}\p_k$.
Next, we claim more generally that
$K(A_k,\p_{k+1})\subset K(A_{k-1},\p_{k})$.
This follows because $A_k=l_k^TA_{k-1}l_k$.  The three operators
$l_k$, $l_k^T$ and $A_{k-1}$ all map $K(A_{k-1},\p_k)$ into itself
since $\g_k$ and $\p_k$ are both already proven to lie in this
space.  Thus, $A_k$ maps $K(A_{k-1},\p_k)$ into itself.

Thus, we have shown $K(A_k,\p_{k+1})\subset K(A_{k-1},\p_{k})$.
To conclude the proof, we must show that it is a proper subspace.
We claim that $K(A_k,\p_{k+1})\subset \p_k^\perp$.  Observe first
that $\p_{k+1}\in\p_k^{\perp}$; this follows immediately from
\eref{eq:pgorthog}.  Next, it is obvious from the
definition of $l_k$ that $\p_k$ is a right eigenvector of $l_k$.  Furthermore,
$\p_k$ is a right eigenvector of $l_k^TA_{k-1}$, as we see from the following
algebra:
\begin{eqnarray*}
l_k^TA_{k-1}\p_k & = & \left(I+\frac{\g_k\p_k^T}{\p_k^T\p_k}\right)A_{k-1}\p_k \\
& = & A_{k-1}\p_k +\g_k\frac{\p_k^TA_{k-1}\p_k}{\p_k^T\p_k} \\
& = & A_{k-1}\p_k + \g_k/\alpha_k \\
& = & A_{k-1}\p_k + (\p_k-\alpha_kA_{k-1}\p_k)/\alpha_k \\
& = & \p_k/\alpha_k.
\end{eqnarray*}
The statement under consideration
$K(A_k,\p_{k+1})\subset \p_k^\perp$ can be rewritten as
the equation $\p_k^TA_k^i\p_{k+1}=0$ for all $i$, i.e.,
$\p_k^T(l_k^TA_{k-1}l_k)^i\g_k=0$.  But $(l_k^TA_{k-1}l_k)^i$ can be 
factored as products of $l_k^T$ and $A_{k-1}l_k$, and we have just
proved that $\p_k^T$ is a left eigenvector of both of these operators.
Therefore, $\p_k^T(l_k^TA_{k-1}l_k)^i\g_k=\mbox{scalar}\cdot\p_k^T\g_k=0$.

Therefore, we have proved that 
$$K(A_k,\p_{k+1})\subset K(A_{k-1},\p_{k})\cap \p_k^{\perp}.$$
Thus, to show that $K(A_k,\p_{k+1})$ is a proper subset of 
$K(A_{k-1},\p_{k})$, it suffices to show that $K(A_{k-1},\p_k)$ is
not a subspace of $\p_k^{\perp}$.  But this is obvious, since the
former contains $\p_k$ while the latter does not.
\end{proof}

This theorem proves finite termination of Algorithm 1: 
the dimension of the invariant subspace at iteration 0 is at
most $n$, and the dimension shrinks by at least 1 each iteration,
so therefore the algorithm terminates in at most $n$ iterations.

More strongly, if the coefficient matrix $A_0$ has at most $s$ distinct
eigenvalues, then Algorithm 1 terminates in at most $s$ iterations,
since any vector lies in an invariant subspace of dimension at most $s$
for such a matrix.

Finally, the above theorem suggests that Algorithm 1 
converges superlinearly.  We recall the following two
facts (see \cite{GVL}): 
the steepest descent algorithm applied to a convex quadratic
function converges at a rate proportional to the condition number
of the matrix.  Furthermore, the condition number of a matrix
acting on a subspace can never exceed (and is usually less than) the
condition number of the matrix acting on the whole space, a consequence
of the Courant-Fisher minimax theorem.  Thus,
we see that Algorithm 1 consists of steepest descent in ever smaller
invariant subspaces, so the effective condition number of the matrix
decreases (or at least, does not increase) each iteration and hence
the convergence rate is expected to be superlinear.

\section{Equivalence to linear conjugate gradient}
\label{sec:linearcg}
Again, we assume for this section that
$f(\x)=\x^TA\x/2-\b^T\x$,
where $A\in\R^{n\times n}$ is symmetric
and positive definite.  We assume again that the line search
is exact.  We prove that Algorithm 2 is equivalent to linear
conjugate gradient.   For the sake of completeness, let
us write linear conjugate gradient in its usual form as follows.  Let $\x_0$
be given.
\begin{tabbing}
++++\=++\=++\=\kill
\> {\bf Algorithm Linear-CG} \\
{[1]} \> $\r_0:=\b-A\x_0$; \\
{[2]} \> for $k=1,2,\ldots$ \\
{[3]} \>\>if $k=1$ \\
{[4]} \>\>\> $\n_1=\r_0;$ \\
{[5]} \>\>else \\
{[6]}\>\>\> $\beta_k = \r_{k-1}^T\r_{k-1}/(\r_{k-2}^T\r_{k-2});$ \\
{[7]}\>\>\> $\n_k=\beta_k\n_{k-1}+\r_{k-1};$ \\
{[8]} \>\> end \\
{[9]} \>\> $\alpha_k=\r_{k-1}^T\r_{k-1}/(\n_k^TA\n_k);$ \\
{[10]} \>\> $\x_k=\x_{k-1}+\alpha_k\n_k$; \\
{[11]} \>\> $\r_k=\r_{k-1}-\alpha_kA\n_k$; \\
{[12]}\> end
\end{tabbing}
Well known properties of Linear-CG are that $\r_k=\b-A\x_k=-\nabla f(\x_k)$ 
and that
the $\r_k$'s are mutually orthogonal (see \cite{GVL}).
We claim that
Algorithm 2 and Algorithm Linear-CG are equivalent with the
following relationships among the variables:
$\p_k=\r_{k-1}$; $\m_k=\n_k$; $\g_k=\r_k$, and $\alpha_k$ is the
same between the algorithms. This equivalence is proved by
induction.  For the $k=1$ case, it is clear that 
$\p_1=\m_1=\n_1=\r_0$ and $\g_1=\r_1$.  For $k>1$, we see that
\begin{eqnarray*}
\p_k&=&-l_{k-1}^T\cdots l_1^T\nabla f(\x_{k-1}) \\
&=&\left(I+\frac{\g_{k-1}\p_{k-1}^T}{\p_{k-1}^T\p_{k-1}}\right)
\cdots\left(I+\frac{\g_1\p_1^T}{\p_1^T\p_1}\right)\r_{k-1} 
\\
& = &
\left(I+\frac{\r_{k-1}\r_{k-2}^T}{\r_{k-2}^T\r_{k-2}}\right)
\cdots\left(I+\frac{\r_1\r_0^T}{\r_0^T\r_0}\right)\r_{k-1} \\
&=&\r_{k-1}.
\end{eqnarray*}
The second and third line both involved application of the
induction hyptohesis, and the last line follows because all
terms drop out from the product with $\r_{k-1}$
except the identity because the $\r_i$'s are
mutually orthogonal.

Next, we show by induction that $\n_k=\m_k$.
Observe from step [7] of Linear-CG that 
$\n_k-\beta_k\n_{k-1}=\r_{k-1}$ while
\begin{eqnarray*}
\m_k-\beta_k\m_{k-1} & = &
\m_k-\left(\frac{\r_{k-1}^T\r_{k-1}}{\r_{k-2}^T\r_{k-2}}\right)\m_{k-1} \\
&=&
l_1\cdots l_{k-1}\p_k-\left(\frac{\r_{k-1}^T\r_{k-1}}{\r_{k-2}^T\r_{k-2}}\right)
l_1\cdots l_{k-2}\p_{k-1} \\
&=&l_1\cdots l_{k-2}\left(l_{k-1}\r_{k-1}-\left(\frac{\r_{k-1}^T\r_{k-1}}{\r_{k-2}^T\r_{k-2}}\right)
\r_{k-2}\right) \\
&=&
l_1\cdots l_{k-2}\left(\r_{k-1}+\frac{\r_{k-2}\r_{k-1}^T\r_{k-1}}{\r_{k-2}^T\r_{k-2}}
-\left(\frac{\r_{k-1}^T\r_{k-1}}{\r_{k-2}^T\r_{k-2}}\right)\r_{k-2}\right)
 \\
&=&l_1\cdots l_{k-2}(\r_{k-1}) \\
&=&\r_{k-1}.
\end{eqnarray*}
In the above derivation, we applied the induction hypothesis, the
definition of $l_{k-1}$, and, for the last line, again the fact that 
the $\r_i$'s are mutually orthogonal.
This equation proves that $\n_k-\beta_k\n_{k-1}=\m_k-\beta_k\m_{k-1}$,
hence the sequence of $\m_k$'s and $\n_k$'s are equal. Finally,
we must claim that $\g_k=\r_k$.  Again, this follows from step [6] of
Algorithm 2 and the
orthogonality of the $\r_k$'s.

\section{The secant condition}
In this section we drop the assumption that $f$ is quadratic
but continue to assume that it is $C^1$.  We prove that 
if the line search is exact, then
Algorithm 2 satisfies the secant condition, which states
$$H_{k+1}\y_{k}=\m_{k}$$
where $H_{k+1}=l_1\circ\cdots \circ l_{k}\circ l_{k}^T\circ\cdots
\circ l_1^T$, that is, the operator that carries 
$-\nabla f(\x_k)$ to $\m_{k+1}$, and
$\y_{k+1}=\nabla f(\x_{k+1})-\nabla f(\x_k)$.
The secant condition is usually stated as the
requirement that $H_{k+1}\y_k=\alpha_k\m_{k}$ \cite{NocedalWright}.  
The scaling
factor, however, is inconsequential because the algorithm can
be equivalently presented 
with a different scaling of $H_k$; that scaling would be canceled
in the line search, which would carry out the reciprocal scaling.

It should be noted that the best known secant algorithms
including DFP and BFGS satisfy the secant condition regardless
of whether the line search
is exact, so Algorithm 2 differs from these algorithms
in this respect.

Checking the secant condition is fairly straightforward
algebra as follows. 
It follows from steps [2] and [6] that
$l_{k-1}^T\circ\cdots\circ l_1^T(\nabla f(\x_{k+1}))=-\g_k$
and $l_{k-1}^T\circ\cdots\circ l_1^T(\nabla f(\x_{k}))=-\p_k$,
hence
$$l_{k-1}^T\circ\cdots\circ l_1^T(\nabla f(\x_{k+1})-\nabla f(\x_{k}))
=\p_k-\g_k.$$
Next, applying $l_k^T$ yields:
\begin{eqnarray*}
l_k^Tl_{k-1}^T\circ\cdots\circ l_1^T(\nabla f(\x_{k+1})-\nabla f(\x_{k}))
&=&
l_k^T(\p_k-\g_k) \\
&=&\left(I+\frac{\g_k\p_k^T}{\p_k^T\p_k}\right)(\p_k-\g_k) \\
&=&\p_k-\g_k+\frac{\g_k\p_k^T\p_k}{\p_k^T\p_k}-\frac{\g_k\p_k^T\g_k}
{\p_k^T\p_k} \\
&=&\p_k,
\end{eqnarray*}
where, to obtain the last line, we invoked \eref{eq:pgorthog} since
the line search is exact.  Next, since $\p_k$ is an eigenvector
of $l_k$ with eigenvalue 1 (again using the fact that the
line search is exact so $\g_k^T\p_k=0$),
$$l_kl_k^Tl_{k-1}^T\circ\cdots\circ l_1^T(\nabla f(\x_{k+1})-\nabla f(\x_{k}))
=\p_k.$$
Finally, applying $l_1\cdots l_{k-1}$ to both  sides and applying
statement [3] yields the desired result.

\section{Computational results (preliminary)}

In this section we compare Algorithm 2 to BFGS, DFP, Polak-Ribi\`ere
conjugate gradient (CG-PR+), and Fletcher-Reeves conjugate 
gradient (CG-FR).  
Refer to \cite{NocedalWright} for information about all of these
algorithms.  
 In this section we denote Algorithm 2 as
SDICOV for ``steepest descent with iterated change of variables.''
We report only the number of iterations.  The BFGS and
DFP algorithms are implemented using product form rather than explicit
formation of $H_k$.  This means that, like SDICOV, the number of
operations
and storage requirement for the $k$th iteration is $O(kn)$ plus a function
and gradient evaluation (plus additional function and gradient evaluations
in the line search).  In contrast, CG-PR+ and CG-FR require only $O(n)$
storage and $O(n)$ operations per iteration.  Therefore, the iteration
counts reported here partially hide the greater efficiency of CG-PR+ 
and CG-FR. 

The first test is the nonconvex {\em distance geometry} problem
\cite{Hendrickson},
a nonlinear least squares problem.  There are $n$ particles
in $\R^2$ whose positions are unknown.  One is given the interparticle
distances for some subset of  possible pairs of particles.  The
problem is to recover the coordinates from these distances.  Thus,
the unknowns are $\x_3,\ldots,\x_n$, positions of particles 3 to $n$,
each a vector in $\R^2$.
To remove
degenerate degrees of freedom,
we assume the positions of particles 1 and 2 are fixed.
The objective function is
$$f(\x_3,\ldots,\x_n)=\sum_{(i,j)\in E} (||\x_i-\x_j||^2 - d_{ij}^2)^2$$
where $E$ denotes the subset of $\{1,\ldots,n\}^2$ of pairs whose
distance is given and $d_{ij}$ denotes the given distance.
This problem has multiple local minima
(indeed, global minimization of this function
is known to be NP-hard), so the testing procedure must account for
the possibility that
that different algorithms could converge to different
minimizers, which could skew iteration counts.
To avoid this possibility, we constructed instances with a known
global minimizer (by first selecting the positions randomly, and then computing
the interpair
distances from those positions).  Then we initialized the algorithm
fairly close to the global minimizer so that all algorithms would
fall into the same basin.  The line search procedure is inexact:
it uses bisection
with a termination criterion that $|\phi'(\alpha)|\le 0.2|\phi'(0)|$.
The convergence tolerance is a relative reduction in the norm of the
gradient of $10^{-5}$.  Two sizes were tried, namely 10
particles ($n=16$) and 100 particles ($n=196$).  For each
problem size, four trials were run, and the number of iterations
over the trials was averaged.  The results are summarized
in Table~\ref{table:distg}.  For the smaller problem SDICOV was worse
than BFGS or DFP, but for the larger problem, the three algorithms
have similar performance.  The two versions of conjugate gradient
are slower.

\begin{table}
\begin{center}
\caption{Results of distance geometry trials}
\label{table:distg}
\begin{tabular}{lrr}
\hline
\multicolumn{1}{l}{Algorithm} & \multicolumn{2}{c}{Ave.~no.~iterations} \\
& $n_{\rm particle}=10$ & $n_{\rm particle}=100$ \\
\hline
SDICOV & 34 & 76 \\
BFGS & 20 & 75 \\
DFP & 24 & 80 \\
CG-PR+ & 93 & 107 \\
CG-FR & 146 & 161 \\
\hline
\end{tabular}
\end{center}
\end{table}

The second test is a larger class of problems, namely, a finite
element mesh improvement problem.  Given a subdivision of a
region $\Omega\subset \R^3$ into tetrahedra, the problem under consideration
is to displace the nodes of the tetrahedra in such a way as to improve
the overall quality of the mesh.  There are several measures of quality;
we use the ratio of the volume of the tetrahedra to
the cube of one of its side lengths.  The minimum such ratio over all 
tetrahedra is a measure of the mesh quality (the closer to 0, the worse
the mesh).  The details of our method are in \cite{Sastry}.  Briefly,
we smooth this nonsmooth unconstrained problem (nonsmooth because it
is maximization of a minimum) by introducing an auxiliary
variable standing for the minimum ratio and constraints to enforce its
minimality.
The smoothed constrained problem is the solved with a barrier function
approach.  Ultimately, the problem reduces again to an unconstrained
problem, except the objective function is a smoothed version of the
original that involves the logarithms of the ratios.

There is a second source of nondifferentiability that remains
in the problem due to parametrization of the boundary.  For interior
nodes in the mesh, the variables in the optimization problem are
its $(x,y,z)$ coordinates.  For nodes on the boundary, however, the
variables are the $(u,v)$ or $t$ parametric coordinates of the boundary
surface.  We wish to allow nodes on the boundary to move from one parametric
patch of a boundary surface to another; 
such movement introduces a nondifferentiable jump
in the objective function.  If the boundary surfaces are smooth, 
it would be possible
in principle to come up with smooth local parametrizations that would
circumvent this difficulty, but we have not done so.

The line search is again based on bisection and enforces the inequality
$|\phi'(\alpha)|\le 0.7|\phi'(0)|$.  It needs a safeguard,
since a step too large can invert a tetrahedron, thus sending the above ratio
to a negative number and hence making the logarithm undefined.
The initial point for the optimization routine
is the mesh produced by the QMG mesh generator \cite{MitchVava:QMG}.

We tested three problems: a mesh of
a cylinder, of a cube with a large spherical
cavity, and of a tetrahedron with a small octahedral cavity.  
For this third problem, each boundary surface is a single flat
parametric patch, so the problem is differentiable because there
are no parametric jumps.
The results of this
test are shown in Table~\ref{table:meshimprove}.  This problem is again
nonconvex and probably has many local minima.  In this test case, we
did not have a means to ensure that the different algorithms find the
same minimizer.  The algorithms, however, returned solutions with comparable
objective function values (when they succeeded).    
BFGS and DFP failed in every case in the sense that they terminated
due to stagnation prior to satisfaction of the convergence termination
criterion.  Our test for stagnation was four successive iterations without
significant reduction in either the function value or gradient norm.
Prior to stagnation, there was generally slow progress in these
algorithms; for example, the
stagnation test required
127 iterations for BFGS and 126 for DFP in the
cylinder case to activate. The two conjugate gradients also 
sometimes failed due to
stagnation; CG-PR+ also failed once
for producing a search direction that was not a descent direction.

It must be pointed out that we have not implemented a restart strategy
for either BFGS or DFP.  Most modern implementations would have such
a strategy, and this would presumably ameliorate the difficulty with
slow progress.

\begin{table}
\begin{center}
\caption{Results of the optimization algorithms on the mesh improvement
problem.  A missing entry indicates failure of the iteration.}
\label{table:meshimprove}
\begin{tabular}{lrrr}
\hline
& \multicolumn{3}{c}{No.~of iterations} \\
  & Cylinder & Large cavity & Small cavity \\
\hline
$n$ & 7380 & 8775 &   4254 \\
SDICOV & 22 & 87 & 373  \\
BFGS & --- & --- & --- \\
DFP & --- & --- & --- \\
CG-PR+ & 34 & --- & --- \\
CG-FR & 81 & --- & --- \\
\hline
\end{tabular}
\end{center}
\end{table}

\section{Concluding remarks}
We propose a new iterative method for unconstrained minimization.
The algorithm is based on steepest descent after a linear change
of coordinates.  
It is a secant method if the line search is exact.
It always produces a descent direction except 
in the case that the line search produces a certain
degenerate result.

In practice,
the new method works well on two test cases.  
Our preliminary results hint
that, if restarting is not used, the new algorithm in practice is 
sometimes more robust than  BFGS and DFP.

\bibliography{../../Bibfiles/icov}
\bibliographystyle{plain}
\end{document}